\documentclass{amsart}
\numberwithin{equation}{section}

\newtheorem{thm}{Theorem}[section]

\newtheorem{remark}[thm]{Remark}

\begin{document}
\title[Inverse problem of determining the order  \dots]
{Inverse problem of determining the order of the fractional derivative in the Rayleigh-Stokes equation}

\author[Ravshan Ashurov, Oqila Mukhiddinova, \hfil \hfilneg] {Ravshan Ashurov, Oqila Mukhiddinova}  

\address{Ravshan Ashurov \newline
V.I. Romanovskiy Institute of Mathematics,\\
Uzbekistan Academy of Science, \\
University str.,9, Olmazor district,\\
Tashkent, 100174, Uzbekistan}
\email{ashurovr@gmail.com}

\address{Oqila Mukhiddinova \newline
Tashkent University of Information Technologies,\\
108 Amir Temur Avenue, \\
Tashkent, 100200, Uzbekistan}
\email{oqila1992@mail.ru}

\subjclass[2000]{} 
\begin{abstract}In recent years, much attention has been paid to the study of forward and inverse problems for the Rayleigh-Stokes equation in connection with the importance of this equation for applications. This equation plays an important role, in particular, in the study of the behavior of certain non-Newtonian fluids. The equation includes a fractional derivative of order $\alpha$, which is used to describe the viscoelastic behavior of the flow. In this paper, we study the behavior of the solution of such equations depending on the parameter $\alpha$. In particular, it is proved that for sufficiently large $t$ the norm $||u(x,t)||_{L_2(\Omega)}$ of the solution decreases with respect to $\alpha$. Moreover the inverse problem of determining the order of the derivative $\alpha$ is solved uniquely.

\end{abstract}

\maketitle \numberwithin{equation}{section}
\newtheorem{theorem}{Theorem}[section]
\newtheorem{corollary}[theorem]{Corollary}
\newtheorem{lemma}[theorem]{Lemma}

\newtheorem{problem}[theorem]{Problem}
\newtheorem{example}[theorem]{Example}
\newtheorem{definition}[theorem]{Definition}
\allowdisplaybreaks

\section{Introduction}

A fractional model of a generalized fluid flow of the second kind can be represented as a Rayleigh-Stokes problem with a fractional time derivative (see, for example, \cite{Bazh}):
\begin{equation}\label{probIN}
	\left\{
	\begin{aligned}
		&\partial_t u(x,t)  -(1+\gamma\, \partial_t^\alpha)\Delta u(x,t) = f(x, t),\quad x\in \Omega, \quad 0< t \leq T;\\
		&u(x, t) = 0, \quad x\in \partial \Omega, \quad 0 < t \leq T;\\
		&u(x, 0)= \varphi(x), \quad x\in \Omega,
	\end{aligned}
	\right.
\end{equation}
where $1/\gamma>0$ is the fluid density, a fixed constant, $f(x,t)$ is the source term  and $\varphi(x)$ is the initial data, $\partial_t= \partial/\partial t$, and $\partial_t^\alpha$ is the Riemann-Liouville fractional derivative of order $\alpha \in (0,1)$ defined by (see, e.g. \cite{KilSriTru}):
\begin{equation}\label{RL}
	\partial_t^\alpha h(t)= \frac{d}{dt} \int\limits_0^t \omega_{1-\alpha}(t-s) h(s)ds, \quad \omega_{\alpha}(t)=\frac{t^{\alpha-1}}{\Gamma(\alpha)}.
\end{equation}
Here $\Gamma(\sigma)$ is
Euler's gamma function. Based on physical considerations, usually the authors consider this problem in the domain $\Omega\subset \mathbb{R}^N$, $N=1,2,3$, and for $N>1$ it is assumed that the boundary $\partial\Omega$ of the domain $\Omega$ is sufficient smooth (see e.g. \cite{Bazh}, \cite{Duc}, \cite{Luc1}).

This equation for $\alpha=1$ is called the Haller equation and is a mathematical model of water movement in capillary-porous media, which include soils (see, for example, in \cite{Chud}, formulas (1.4) and (1.84) on p. 137 and 158, \cite{Nakh1}, formula (9.6.4) on p. 255, and \cite{Nakh2}, formula (2.6.1) on p. 59). Various initial-boundary value problems in this case were studied, for example, in the books \cite{Nakh1} and \cite{Nakh2} (see also the literature therein).

In recent years, the Rayleigh-Stokes problem (\ref{probIN}) has received much attention due to its importance for applications (see, for example, \cite{Tan1} - \cite{AshurovVaisova}). An overview of work in this direction can be found in Bazhlekova et al. \cite{Bazh} (see also \cite{AshurovVaisova}). Here we only note the main directions in which research was carried out.

1) In order to get an idea of the behavior of the solution of this model, an exact solution has been obtained in some special cases; see e.g. \cite{Fet}, \cite{Shen}, \cite{Zhao};

2) The Sobolev regularity of the homogeneous problem was studied in the fundamental work of Bazhlekov et al. \cite{Bazh} (for the case $f(x,t)\neq 0$ see \cite{AshurovVaisova});
\

3) A number of specialists have developed efficient and optimally accurate numerical algorithms for solving the problem (\ref{probIN}). A review of some works in this direction is contained in the above-mentioned paper \cite{Bazh}. See also recent papers \cite{Le}, \cite{Dai} and references therein;

4) Many works are devoted to the study of the inverse problem of determining the right-hand side of the Rayleigh-Stokes equation (see, for example, \cite{Duc},
\cite{Tran1}, \cite{Tran2},  and the bibliography cited there). Since this inverse problem is Hadamard ill-posed, various regularization methods and numerical methods for finding the right-hand side of the equation are proposed in these works;

5) If the initial condition $u(x, 0)= \varphi(x)$ in the problem (\ref{probIN})  is replaced by $u(x, T)= \varphi(x)$, then the resulting problem is called \emph{backward problem}. The backward problem for the Rayleigh-Stokes equation is of great importance and is aimed at determining the previous state of the physical field (for
for example, at t = 0) based on his current information (see, for example, \cite{Luc1}, \cite{Luc2} for the case $N\leq 3$ and \cite{AshurovVaisova} for an arbitrary $N$). However, this problem (as well as the inverse problem of finding the right-hand side of the equation) is ill-posed according to Hadamard. Therefore, the authors of \cite{Luc1}, \cite{Luc2} proposed various regularization methods and tested these methods using numerical experiments.
In the paper \cite{AshurovVaisova}, along with other questions, problem (\ref{probIN})  is investigated by taking the non-local condition $u(x, T)= u(x,0) + \varphi(x)$ instead of the initial condition. The authors proved that the non-local problem is well-posed in the sense of Hadamard: i.e. the unique solution exists, and the solution continuously depends on the initial data and on the right-hand side of the equation.

It is well known that the Rayleigh-Stokes problem (\ref{probIN}) plays an important role in the study of the behavior of some non-Newtonian fluids. The fractional derivative $\partial_t^\alpha$ is used in the equation (\ref{probIN}) to describe the behavior of viscoelastic flow (see, for example, \cite{Fet}, \cite{Shen}). Obviously, knowing the exact value of this parameter is interesting not only theoretically; it is necessary for more adequate modeling of the physical process. But this parameter is often unknown and difficult to measure directly. Therefore, it is undoubtedly relevant to study the inverse problem to determine this
a physical quantity from some indirectly observed information about the solutions. Such an inverse
the problem would be not only of theoretical interest, it would help to more accurately determine the solutions of the initial-boundary value problem and study the properties of the solutions.

The present work is devoted to the study of this new inverse problem for the Rayleigh-Stokes equation. To solve the inverse problem, it is necessary to set an additional condition. A natural requirement for such a condition is that it must ensure both the existence and uniqueness of the unknown parameter. In this paper, along with other results, it is proved that the additional condition $||u(x, t_0)||^2_{L_2(\Omega)}=d_0$ for sufficiently large $t_0$ just satisfies these requirements.

Another very important problem in general for any fractional order equation is the study of the dependence of the behavior of the solution of the initial-boundary value problem on the order of the fractional derivative. To the best of our knowledge, this problem has not yet attracted sufficient attention from researchers. In this work, when solving the inverse problem, an interesting fact was discovered: if, for example, the solution norm $||u(x, t_0)||_{L_2(\Omega)}$  is considered as a function of the parameter $\alpha$, then it is a decreasing function. In other words, the norm acquires its maximum value when the order of the fractional derivative is close to zero, and its minimum value - when this parameter is close to one.

The work consists of 7 Sections and Conclusion. In the next section, we present the exact formulation of the problem under study and formulate the main results of the work. The proofs of the main results are based on an estimate for the derivative of the function $B_\alpha(\lambda, t)$, which is defined in Section 3. In the same section, the derivative $\partial_\alpha B_\alpha(\lambda, t)$ is calculated and the Main Lemma about the sign of this derivative is formulated. Section 4 contains all auxiliary lemmas that are needed in what follows. The proof of the Main Lemma is contained in Section 5. Section 6 contains proofs of the main theorems of the paper. The paper considers the abstract Rayleigh-Stokes equation with a self-adjoint operator $A$. Section 7 gives examples of the operator $A$ for which the main results of the paper are valid. In the same section, various variants of the additional condition are formulated.
The paper ends with the section Conclusion.
\section{Problem statements and results}

Let $\Omega\subset \mathbb{R}^N$ be an arbitrary bounded domain with a piecewise smooth boundary $\partial\Omega$ and  $A$ be a positive self-adjoint extension in $L_2(\Omega)$ of the Laplace operator generated with the homogeneous Dirichlet boundary condition. Then $D(A)=W_2^2(\Omega) \cap \dot{W}_2^1(\Omega)$, where $\dot{W}_2^1(\Omega)$ is the closure in the norm of the Sobolev space ${W}_2^1(\Omega)$ of the set of functions from $C^2(\Omega)$ that vanish on the boundary of the domain $\Omega$ (see e.g. \cite{Ber}).

If we denote by $\{v_k(x)\}$ the complete orthonormal in $L_2(\Omega)$ system of eigenfunctions and by $\{\lambda_k\}$ the set of positive eigenvalues  $0<\lambda_1 <\lambda_2\leq \lambda_3\leq \cdot\cdot\cdot\rightarrow +\infty$, of the spectral problem
\begin{equation}\label{spectral}
	\left\{
	\begin{aligned}
		&-\Delta v(x) = \lambda v(x),\quad x\in \Omega;\\
		&v(x)=0, \quad x\in \partial \Omega,
	\end{aligned}
	\right.
\end{equation}
then the action of the operator $A$ can be written as:
\[
A f(x)= \sum\limits_{k=1}^\infty \lambda_k f_k v_k(x),\quad f\in D(A).
\]
Here and below, $f_k$ denotes the Fourier coefficients of $f\in L_2(\Omega)$ with respect to the system of eigenfunctions $\{v_k(x)\}$: $f_k=(f,v_k)$, $(\cdot, \cdot)$ is the inner product in $L_2(\Omega)$.

Our study is based on the classical method of separation of variables. Consequently, from the Laplace operator in the Rayleigh-Stokes problem (\ref{probIN}), we need only the discreteness of the spectrum and the completeness of the systems of eigenfunctions in $L_2(\Omega)$. Therefore, wishing to cover more general elliptic operators, we formulate the Rayleigh-Stokes problem in an abstract form.

Let $H$ be a separable Hilbert space with the scalar product $(\cdot, \cdot)$ and the norm $||\cdot||$. Consider an arbitrary unbounded positive self-adjoint operator $A$ in $H$ with the domain of definition $D(A)$ and having a compact inverse. Let us define a complete system of orthonormal
eigenfunctions by $\{v_k\}$ and a countable set of positive
eigenvalues by $\lambda_k:$ $0<\lambda_1\leq\lambda_2 \cdot\cdot\cdot\rightarrow +\infty$.

For a vector-valued functions (or simply functions)
$h: \mathbb{R}_+\rightarrow H$, the Riemann-Liouville fractional derivative of order $0<\alpha< 1$ is defined in the same way as (\ref{RL}) (see, e.g. \cite{Liz}).
Finally, let $C((a,b); H)$ stand
for a set of continuous in $t\in (a,b)$ functions $u(t)$  with
values in $H$.

Consider the Rayleigh-Stokes problem in an abstract form
\begin{equation}\label{probAbstract}
	\left\{
	\begin{aligned}
		&\partial_t u(t)  + (1+\gamma\, \partial_t^\alpha)A u(t) = 0,\quad 0< t \leq T;\\
		&u(0)= \varphi,
	\end{aligned}
	\right.
\end{equation}
where $\gamma>0$, $\alpha\in (0,1)$ and $\varphi\in H$.

Assume that the order $\alpha$ of the fractional derivative in this problem is unknown. Consider the inverse problem of determining this parameter. The homogeneity of the equation greatly facilitates the application of the method proposed in this paper.

To find $\alpha$, it is necessary to set an additional condition. The natural requirement for this condition is that it must guarantee both the existence and the uniqueness of the unknown parameter. Before formulating this condition, we note that the solution to problem (\ref{probAbstract}) naturally depends on $\alpha$, although we denote it by $u(t)$.

The additional condition we propose has the form
\begin{equation}\label{ad_con}
	U(t_0, \alpha)\equiv	||u(t_0)||^2= d_0,\quad t_0\geq T_0,
\end{equation}
where $T_0$ is defined later. Problem (\ref{probAbstract}) together with the additional condition (\ref{ad_con}) will be called \textit{the inverse problem}.

This inverse problem can be
interpreted as follows: can we uniquely identify the unknown order of the fractional derivative $\alpha$ if, as additional information, we have the mean square value of the solution to the Rayleigh-Stokes problem (\ref{ad_con}) at a fixed time instance $t_0$?

\begin{definition}\label{def} A pair $\{u(t), \alpha\}$ of the function $u(t)$ and the parameter  $\alpha$ with the properties
	\begin{enumerate}
		\item
		$\alpha \in (0, 1)$,
		\item$u(t)\in
		C([0,T]; H)$,
		\item
		$\partial_t u(t), \,  A u(t), \, \partial_t^\alpha Au(t)\in C((0,T); H)$
	\end{enumerate}
	and satisfying all the conditions of
	problem (\ref{probAbstract}) - (\ref{ad_con}) is called
	\textbf{the solution} of inverse problem.
\end{definition}

Before stating the main result of the work on the solution of the inverse problem, we formulate an auxiliary assertion, which is also of independent interest.

\begin{theorem}\label{norm}
	There exists a number $T_0=T_0(\gamma, \alpha, \lambda_1)\geq 1$ such that for every $t_0\geq T_0$ and every $\varphi \in H$ the function $U(t_0, \alpha )$ strictly decreasing in $\alpha\in (0,1)$.	\end{theorem}

This theorem establishes the dependence of the norm of the solution of the Rayleigh-Stokes problem on the order of the fractional derivative.  Namely, the norm acquires its maximum value when the order of the fractional derivative is close to zero, and its minimum value - when this parameter is close to one. There are only a few works where such a dependence of the solution of initial-boundary value problems for subdiffusion equations has been studied (see, for example, \cite{AlimovAshurov1} and \cite{AlimovAshurov2}).

Here is the main resul of the paper.

\begin{theorem}\label{main}
	Let $\varphi \in H$ and $t_0\geq T_0$. Then for the inverse problem to have a solution $\{u(t), \alpha\}$ it is necessary and sufficient that the condition
	\begin{equation}\label{U_con}
		\inf_{\alpha\in (0,1)}U(t_0, \alpha)\leq d_0\leq 	\sup_{\alpha\in (0,1)}U(t_0, \alpha)
	\end{equation}
	be satisfied.
\end{theorem}

\begin{remark}
	Theorem defines the unique $ \alpha $ from (\ref {ad_con}). Hence, if we
	define the norm $U(t_0, \alpha)$ at another time instant $ t_1\geq T_0 $
	and get a new $\alpha_1$, i.e. $U(t_1, \alpha_1)=d_1$, then from the
	equality $U(t_0, \alpha_1)=d_0$, by virtue of the theorem, we obtain
	$\alpha_1 = \alpha$.
\end{remark}
The inverse problem of determining the order of the fractional derivative for the subdiffusion equations, the fractional-wave equation, and mixed-type equations has been studied by many authors. An overview of works published up to 2019 can be found in \cite{LiLiu}. After this paper, a lot of works have been devoted to this subject, and a brief review of these works can be found in works \cite{AshurovMatzametki} and \cite{AshurovSitnik}.

Results close to Theorem \ref{main} were obtained in work \cite{AshurovUmarov1} for the subdiffusion equation by studying the behavior of the derivative of the Mittag-Leffler function $E_{\alpha, 1}$ with respect to the parameter $\alpha$. Works \cite{AshurovUmarov2} and \cite{AshurovUmarov3} are devoted to the study of the inverse problem of determining the vector order of the fractional derivative for systems of pseudodifferential equations.

\section{Derivative of function $B_\alpha(\lambda, t)$ with respect to parameter $\alpha$}
When solving the Rayleigh-Stokes problem using the method of separation of variables, we come to the solution of an ordinary differential equation (see e.g. \cite{Bazh} and \cite{AshurovVaisova}))
\begin{equation}\label{B}
	L y(t)\equiv	y'(t)+\lambda (1+\gamma \partial_t^\alpha)y(t)=0, \,\, t>0,\,\,\lambda>0.
\end{equation}
It is well known that the solution to such an equation can be expressed in terms of the generalized Wright function (see, for example, A.A. Kilbas et. al \cite{KilSriTru}, Example 5.3, p. 289). We also note the work of Pskhu \cite{Pskhu}, where more general equations than (\ref{B}) were studied. From the results of this article, one can obtain a representation of the solution of the Cauchy problem for the equation (\ref{B}), which is very convenient for further research.

But for our reasoning it is convenient to use the fundamental result of work  Bazhlekova, Jin, Lazarov, and Zhou \cite{Bazh} on the solution of equation (\ref{B}). Before formulating the corresponding results, we note that since the equation is homogeneous and linear, to study the property of the solution of the Cauchy problem for equation (\ref{B}), it suffices to consider the Cauchy condition of the form $y(0)=1$.

The authors of \cite{Bazh}, in particular, proved the following lemma.
\begin{lemma}\label{Bazh} Let $B_\alpha(\lambda, t)$ be a solution of the Cauchy problem for equation (\ref{B}) with an initial condition $y(0)=1$. Then
	\begin{enumerate}
		\item
		$B_\alpha (\lambda, 0)=1,\,\, 0<B_\alpha (\lambda, t)<1,\,\, t>0$,
		\item$\partial_t B_\alpha (\lambda, t)<0, \,\, t\geq 0$,
		\item$\lambda B_\alpha (\lambda, t)< C \min \{t^{-1}, t^{\alpha-1}\}, \,\,t>0$,
		\item
		$\int\limits_0^T B_\alpha (\lambda, t) dt \leq \frac{1}{\lambda}, \,\, T>0.$
	\end{enumerate}
	
\end{lemma}

The following representation of function $B_\alpha (\lambda, t)$, also obtained in \cite{Bazh}:
\begin{equation}\label{BInt}
	B_\alpha (\lambda, t)=\int\limits_0^\infty e^{-rt} b_\alpha(\lambda, r) dr,
\end{equation}
where
\[
b_\alpha(\lambda, r)=\frac{\gamma}{\pi} \frac{\lambda r^\alpha \sin \alpha \pi}{(-r+\lambda\gamma r^\alpha \cos \alpha \pi +\lambda)^2+(\lambda \gamma r^\alpha \sin \alpha \pi )^2}.
\]
This implies, in particular, that the function $B_\alpha (\lambda, t)$ does not give a solution to the Cauchy problem for $\alpha=0$, $\alpha=1$, $\gamma=0$ and $\lambda =0$.

Let us present one more result concerning the solution of the general Cauchy problem, established in the same paper \cite{Bazh} (see also \cite{AshurovVaisova}).
\begin{lemma}\label{BazhEquation} The Cauchy problem
	\begin{equation}\label{BCauchyIN}
		y'(t)+\lambda (1+\gamma \partial_t^\alpha)y(t)=0, \,\, t>0,\,\,\lambda>0,\,\, y(0)=y_0,
	\end{equation}
	has the only solution
	\begin{equation}\label{BCauchySolutionIN}
		y(t)=y_0 B_\alpha(\lambda, t).	
	\end{equation}
\end{lemma}

The next main lemma states that if $t_0$ is large enough, then the derivative
$\partial_\alpha B_\alpha (\lambda, t_0)$ is negative for $\lambda\geq \lambda_1>0$, where $\lambda_1$ is the first eigenvalue of the operator $A$.
\begin{lemma}(\textbf{Main Lemma).} Let $\gamma>0$, $\lambda\geq \lambda_1$ and $\alpha\in (0,1)$ be given numbers. There exists a positive number $T_0=T_0(\gamma, \alpha, \lambda_1)\geq 1$ such that for any $t_0$, $T_0\leq t_0\leq T$, one has
	\begin{equation}\label{mainT0}
		\partial_\alpha B_\alpha (\lambda, t_0)<0.
	\end{equation}
	
\end{lemma}

The proof of the lemma is contained in Section 5. Here we calculate the derivative $\partial_\alpha B_\alpha (\lambda, t_0)$. To do this, we write $B_\alpha (\lambda, t_0)$ in a form that is convenient for us.

Let us change the variable $r t_0=\xi$ in the integral for $B_\alpha (\lambda, t_0)$. Then
\begin{equation}\label{BIntNew}
	B_\alpha (\lambda, t_0)=\frac{1}{t_0}\int\limits_0^\infty e^{-\xi} \, b_\alpha\left(\lambda, \frac{1}{t_0}\xi\right) d\xi=t_0^{\alpha-1}
	\int\limits_0^\infty e^{-\xi} b_{\alpha,1}(\lambda, t_0, \xi) d\xi,
\end{equation}
where
\[
b_{\alpha,1}(\lambda,t_0, \xi)=\frac{1}{\pi} \frac{g(\xi, \alpha)}{f^2(\xi, t_0, \alpha)+g^2(\xi, \alpha)},
\]
and
\[
g(\xi,\alpha)=\lambda \gamma \xi^\alpha\sin \alpha \pi, \,\, f(\xi, t_0, \alpha)= -\xi t_0^{\alpha-1}+\lambda \gamma \xi^\alpha \cos \alpha\pi +\lambda t_0^\alpha.
\]

The derivative $\partial_\alpha B_\alpha (\lambda, t_0)$ can be written as the sum of several terms. The first is the result of differentiating the function $t_0^{\alpha-1}$:
\[
I_1=t_0^{\alpha-1} \ln t_0\,
\int\limits_0^\infty e^{-\xi} b_{\alpha,1}(\lambda, t_0, \xi) d\xi.
\]

Let us move on to differentiating the fraction $b_{\alpha,1}(\lambda,t_0, \xi)$. After differentiating the numerator $g(\xi, \alpha)$, we get
\[
I_2=t_0^{\alpha-1}\frac{\gamma}{\pi}\,
\int\limits_0^\infty e^{-\xi}\, \frac{\lambda \xi^\alpha [\ln \xi \sin \alpha\pi +\pi \cos \alpha\pi]}{f^2(\xi, t_0, \alpha)+g^2(\xi, \alpha)} d\xi.
\]
We write the result of differentiating the denominator in the form of the following three integrals:
\[
I_3=-t_0^{\alpha-1} \ln t_0\,
\int\limits_0^\infty e^{-\xi} b_{\alpha,1}(\lambda, t_0, \xi)\frac{2\,f(\xi, t_0,\alpha)(-\xi t_0^{\alpha -1}+\lambda t_0^\alpha)}{f^2(\xi, t_0, \alpha)+g^2(\xi, \alpha)} d\xi,
\]
\[
I_4=-t_0^{\alpha-1} \,
\int\limits_0^\infty e^{-\xi} b_{\alpha,1}(\lambda, t_0, \xi)\frac{2\,f(\xi, t_0,\alpha)\lambda\gamma \xi^\alpha [\ln \xi \cos \alpha\pi -\pi \sin \alpha\pi]}{f^2(\xi, t_0, \alpha)+g^2(\xi, \alpha)} d\xi,
\]
and
\[
I_5=-t_0^{\alpha-1} \,
\int\limits_0^\infty e^{-\xi} b_{\alpha,1}(\lambda, t_0, \xi)\frac{2\, g(\xi, \alpha)\lambda\gamma \xi^\alpha [\ln \xi \sin \alpha\pi +\pi \cos \alpha\pi]}{f^2(\xi, t_0, \alpha)+g^2(\xi, \alpha)} d\xi.
\]
Note here that the first two integrals are the result of differentiating the function $f$, and the third one is the result of differentiating the function $g$.

Now we can write
\begin{equation}\label{derivative}
	\partial_\alpha B_\alpha (\lambda, t_0)=\sum\limits_{k=1}^5 I_k.
\end{equation}

\section{Auxiliary Lemmas}
This section contains some auxiliary statements that are necessary in the proof of the main lemma.

Let us start with a simple statement.
\begin{lemma}\label{gamma} Let $a$ be any real number and $\varepsilon>0$. Then there exists a number $R_0=R_0(a, \varepsilon)$ such that for $R>R_0$ one has
	\[
	\int\limits_R^\infty e^{-r}r^a dr = O(R^{-\varepsilon}).
	\]
\end{lemma}
If $a>-1$ then
\begin{equation}\label{Gamma}
	\Gamma (a+1)=\int\limits_0^\infty e^{-r}r^a dr,
\end{equation}	
and therefore the integral is $o(1)$. But we need a more precise estimate.
\begin{proof} If $a\leq 0$, then for $R>R_0$ one has
	\[
	\int\limits_R^\infty e^{-r}r^a dr\leq R^{\,a} e^{-R} = O(R^{-\varepsilon}).
	\]
	Let $a>0$ and $[a]$ be an integer part of $a$. Integrating by parts $[a]+1$ times, we get
	\[
	\int\limits_R^\infty e^{-r}r^a dr=R^{\,a}e^{-R}+a R^{a-1}e^{-R}+\cdots+ a (a-1)\cdots (a-[a])\int\limits_R^\infty e^{-r}r^{a-[a]-1} dr = O(R^{-\varepsilon}).
	\]
\end{proof}
Recall that in our reasoning, always $\lambda\geq \lambda_1>0$, where $\lambda_1$ is the first eigenvalue of the operator $A$. Let us choose a constant $c_0$ as follows:
\begin{equation}\label{c0}
	\left\{
	\begin{aligned}
		&1< c_0< \frac{1}{2} \big(\gamma +\frac{1}{\lambda_1}\big)^{-1},\,\,\text{if}\,\,\, \gamma +\frac{1}{\lambda_1}< \frac{1}{2}; \\
		&0<c_0^\alpha< \frac{1}{2} \big(\gamma +\frac{1}{\lambda_1}\big)^{-1},\,\,\,\text{if}\,\,\,\gamma +\frac{1}{\lambda_1}\geq \frac{1}{2}.
	\end{aligned}
	\right.
\end{equation}
Note that in the second case $c_0<1$.

\begin{lemma}\label{festimate} Let $\xi\leq c_0\, t_0$. Then
	\[
	\frac{1}{2} \lambda t_0^\alpha\leq f(\xi, t_0, \alpha)\leq \frac{3}{2} \lambda t_0^\alpha.
	\]
	
\end{lemma}

\begin{proof}Let  $\xi\leq c_0\, t_0$. If $c_0>1$, then
	\[
	f(\xi, t_0, \alpha)\geq -\xi t_0^{\alpha-1}-\lambda\gamma\xi^\alpha+\lambda t_0^\alpha\geq -c_0t_0^{\alpha}-\lambda\gamma c_0^\alpha t_0^{\alpha}+\lambda t_0^\alpha\geq
	\]
	\[
	\geq \lambda t_0^\alpha \left[ 1- c_0 \big(\frac{1}{\lambda_1}+\gamma\big)\right]\geq \frac{1}{2} \lambda t_0^\alpha,
	\]
	and
	\[
	f(\xi, t_0, \alpha)\leq \xi t_0^{\alpha-1}+\lambda\gamma\xi^\alpha+\lambda t_0^\alpha\leq c_0t_0^{\alpha}+\lambda\gamma c_0^\alpha t_0^{\alpha}+\lambda t_0^\alpha\leq
	\]
	\[
	\leq \lambda t_0^\alpha \left[ 1+ c_0 \big(\frac{1}{\lambda_1}+\gamma\big)\right]\leq \frac{3}{2} \lambda t_0^\alpha.
	\]
	If $c_0<1$, then
	\[
	f(\xi, t_0, \alpha)\geq-c_0t_0^{\alpha}-\lambda\gamma c_0^\alpha t_0^{\alpha}+\lambda t_0^\alpha\geq
	\]
	\[
	\geq \lambda t_0^\alpha \left[ 1- c_0^\alpha \big(\frac{1}{\lambda_1}+\gamma\big)\right]\geq \frac{1}{2} \lambda t_0^\alpha,
	\]
	and
	\[
	f(\xi, t_0, \alpha)\leq\lambda t_0^\alpha \left[ 1+ c_0^\alpha \big(\frac{1}{\lambda_1}+\gamma\big)\right]\leq \frac{3}{2} \lambda t_0^\alpha.
	\]
\end{proof}
In what follows, it is important for us that for the same $\xi\leq c_0\, t_0$ the function
\[
F(\xi, t_0, \alpha)=f^2(\xi, t_0, \alpha) +g^2(\xi, \alpha) -2f(\xi, t_0, \alpha) (-\xi t_0^{\alpha-1} +\lambda t_0^\alpha)
\]
be negative.
\begin{lemma}\label{Festimate} Let $\xi\leq c_0\, t_0$. Then
	\[
	F(\xi, t_0, \alpha)\leq -\frac{3}{4}\lambda^2 t_0^{2\alpha}.
	\]
	
\end{lemma}

\begin{proof}We have
	\[
	F(\xi, t_0, \alpha)=f\cdot(f+2 \xi t_0^{\alpha-1} -2 \lambda t_0^\alpha)	+g^2=f\cdot(\xi t_0^{\alpha-1} -\lambda t_0^\alpha+\lambda\gamma\xi^\alpha \cos\alpha \pi)	+g^2=
	\]
	\[
	=(\lambda\gamma\xi^\alpha \cos\alpha \pi)^2-(\lambda t_0^\alpha-\xi t_0^{\alpha-1})^2 + (\lambda\gamma\xi^\alpha \sin\alpha \pi)^2=
	\]
	\[
	=(\lambda\gamma\xi^\alpha)^2-(\lambda t_0^\alpha-\xi t_0^{\alpha-1})^2=(\lambda\gamma\xi^\alpha+\lambda t_0^\alpha-\xi t_0^{\alpha-1})(\lambda\gamma\xi^\alpha+\xi t_0^{\alpha-1}-\lambda t_0^\alpha).
	\]
	Let  $\xi\leq c_0\, t_0$ and $c_0>1$. Then by Lemma \ref{festimate},
	\[
	F(\xi, t_0, \alpha)\leq (\lambda\gamma (c_0 t_0)^\alpha+\lambda t_0^\alpha +c_0 t_0 t_0^{\alpha-1})(\lambda\gamma(c_0 t_0)^\alpha+\xi t_0^{\alpha-1}-\lambda t_0^\alpha)\leq -\frac{3}{4}\lambda^2 t_0^{2\alpha}.
	\]
	This estimate is established in exactly the same way for $c_0<1$.
	
\end{proof}

\begin{lemma}\label{b0estimate} Let $\xi\leq c_0\, t_0$. Then
	\[
	\frac{4 \gamma \sin\alpha\pi}{10\pi}\cdot \frac{\xi^\alpha}{\lambda t_0^{2\alpha}}\leq	b_{\alpha,1}(\lambda, t_0, \xi)\leq \frac{4\gamma}{\pi} \cdot \frac{\xi^\alpha}{\lambda t_0^{2\alpha}}.
	\]
	
\end{lemma}

\begin{proof}Let us first estimate the denominator of the fraction $b_{\alpha,1}$. Apply Lemma \ref{festimate} to obtain
	\[
	f^2+g^2\leq \frac{9}{4}\big(\lambda t_0^\alpha\big)^2+\big(\lambda\gamma\xi^\alpha\big)^2\leq \frac{10}{4}\big(\lambda t_0^\alpha\big)^2.
	\]
	Therefore
	\begin{equation}\label{fg}
		\frac{1}{4}\big(\lambda t_0^\alpha\big)^2\leq f^2+g^2\leq \frac{10}{4}\big(\lambda t_0^\alpha\big)^2.
	\end{equation}
	This easily implies the assertion of the lemma.
	
\end{proof}
\begin{lemma}\label{binftyestimate} Let $\xi> c_0\, t_0$. Then
	\[
	b_{\alpha,1}(\lambda, t_0, \xi)\leq\frac{1}{\pi\gamma\lambda \sin\alpha\pi}\cdot \frac{1}{\xi^\alpha}.
	\]
	
\end{lemma}
\begin{proof}By definition
	\[
	\frac{1}{\pi} \frac{g(\xi, \alpha)}{f^2(\xi, t_0, \alpha)+g^2(\xi, \alpha)}\leq \frac{1}{\pi} \cdot\frac{1}{g(\xi, \alpha)}.
	\]	
	
\end{proof}

\begin{lemma}\label{Finftyestimate} Let $\xi> c_0\, t_0$. Then
	\[
	|F(\xi, t_0, \alpha)|\leq 2(\gamma+1)^2 (\lambda^2\xi^{2\alpha} +\xi^2).
	\]
	
\end{lemma}
\begin{proof}Since $t_0>1$, we have
	\[
	|F(\xi, t_0, \alpha)|\leq (\lambda\gamma\xi^\alpha+\lambda t_0^\alpha+\xi t_0^{\alpha-1})^2 \leq (\gamma+1)^2 (\lambda\xi^\alpha +\xi)^2.
	\]	
\end{proof}

\section{Proof of the main lemma}

Let us fix $c_0$ as in (\ref{c0}) and write each of the integrals $I_j$ as the sum of two integrals: over the interval $(0, c_0\, t_0)$ (denoted by $I_{j,0}$) and over the set $(c_0\, t_0, \infty)$ (denoted by $I_{j,\infty}$). Symbolically, this can be written as:
\[
I_j=\int\limits_0^{c_0\, t_0} \cdot +\int\limits_{c_0\, t_0}^\infty \cdot = I_{j,0} + I_{j,\infty}, \, j=1,..., 5.
\]
Thus we have 10 integrals. We select the main one for sufficiently large $t_0$ and estimate its absolute value from below (it turns out that it is negative) and all other integrals from above.

Taking into account the additional factor $\ln t_0$, it is natural to assume that integrals $I_{1,0}$ and $I_{3,0}$ are principal. Let us estimate the sum of these integrals. We have
\[
J_0\equiv I_{1,0}+I_{3,0}=t_0^{\alpha-1} \ln t_0\,
\int\limits_0^{c_0\, t_0} e^{-\xi} b_{\alpha,1}(\lambda, t_0, \xi)\frac{F(\xi, t_0,\alpha)}{f^2(\xi, t_0, \alpha)+g^2(\xi, \alpha)} d\xi.
\]
To estimate $F$, we use Lemma \ref{Festimate}, to estimate the denominator from above, we use (\ref{fg}), and to estimate the function $b_{\alpha, 1}$ from below, we use Lemma \ref{b0estimate}. Then we will obtain
\[
J_0\leq -\frac{3 \gamma \sin \alpha \pi}{25 \lambda\pi}\,t_0^{-\alpha-1} \ln t_0\int\limits_0^{c_0\, t_0} e^{-\xi} \xi^\alpha d\xi.
\]
Since under the conditions of the main lemma $t_0\geq 1$, to improve the estimate, the integral can be taken in the interval $(0, c_0)$. Then the final estimate will look like
\begin{equation}\label{J}
	J_0\leq -
	\frac{C_{0,0}}{\lambda}\,t_0^{-\alpha-1} \ln t_0,
\end{equation}
where $C_{0,0}=C_{0,0}(\gamma, \alpha, \lambda_1)>0$ is a constant.

To estimate the integral
\[
I_{2, 0}=t_0^{\alpha-1}\frac{\gamma}{\pi}\,
\int\limits_0^{c_0 t_0} e^{-\xi}\, \frac{\lambda \xi^\alpha [\ln \xi \sin \alpha\pi +\pi \cos \alpha\pi]}{f^2(\xi, t_0, \alpha)+g^2(\xi, \alpha)} d\xi
\]
from above, we use estimate (\ref{fg}):
\begin{equation}\label{I20}
	\big|I_{2, 0}\big|\leq \frac{4 \gamma}{\pi\lambda}\,t_0^{-\alpha-1}
	\int\limits_0^{c_0 t_0} e^{-\xi}\, \xi^\alpha (|\ln \xi| +1) d\xi\leq C_{2,0}\,t_0^{-\alpha-1},
\end{equation}
where $C_{2,0}=C_{2,0}(\gamma,\alpha, \lambda_1)>0$ is a constant.

Consider the integral
\[
I_{4, 0}=-t_0^{\alpha-1} \,
\int\limits_0^{c_0 t_0} e^{-\xi} b_{\alpha,1}(\lambda, t_0, \xi)\frac{2\,f(\xi, t_0,\alpha)\lambda\gamma \xi^\alpha [\ln \xi \cos \alpha\pi -\pi \sin \alpha\pi]}{f^2(\xi, t_0, \alpha)+g^2(\xi, \alpha)} d\xi.
\]
Application of Lemma \ref{b0estimate} and estimates (\ref{festimate}) and (\ref{fg}) imply
\begin{equation}\label{I40}
	\big|I_{4, 0}\big|\leq \frac{C_4}{\lambda}\,t_0^{-2 \alpha-1}
	\int\limits_0^{c_0 t_0} e^{-\xi}\, \xi^{2\alpha} (|\ln \xi| +1) d\xi\leq C_{4,0}\,t_0^{-2\alpha-1},
\end{equation}
where $C_{4,0}=C_{4,0}(\gamma,\alpha, \lambda_1)>0$ is a constant.

For the integral
\[
I_{5,0}=-t_0^{\alpha-1} \,
\int\limits_0^{c_0 t_0} e^{-\xi} b_{\alpha,1}(\lambda, t_0, \xi)\frac{2\, g(\xi, \alpha)\lambda\gamma \xi^\alpha [\ln \xi \sin \alpha\pi +\pi \cos \alpha\pi]}{f^2(\xi, t_0, \alpha)+g^2(\xi, \alpha)} d\xi,
\]
we again apply Lemma \ref{b0estimate} and estimate (\ref{fg}). Then
\begin{equation}\label{I50}
	\big|I_{5, 0}\big|\leq \frac{C_5}{\lambda}\,t_0^{-3 \alpha-1}
	\int\limits_0^{c_0 t_0} e^{-\xi}\, \xi^{3\alpha} (|\ln \xi| +1) d\xi\leq C_{5,0}\,t_0^{-3\alpha-1},
\end{equation}
where $C_{5,0}=C_{5,0}(\gamma,\alpha, \lambda_1)>0$ is a constant.

Now let us estimate all integrals $I_{j, \infty}$ from above. We note right away that in estimating these integrals we use the estimate $f^2+g^2\geq g^2$.

We start with the integral
\[
J_\infty\equiv I_{1,\infty}+I_{3,\infty}=t_0^{\alpha-1} \ln t_0\,
\int\limits_{c_0\, t_0}^\infty e^{-\xi} b_{\alpha,1}(\lambda, t_0, \xi)\frac{F(\xi, t_0,\alpha)}{f^2(\xi, t_0, \alpha)+g^2(\xi, \alpha)} d\xi.
\]
Apply Lemmas  \ref{binftyestimate} and \ref{Finftyestimate} to obtain (see (\ref{Gamma}))
\[
|J_\infty|\leq  \frac{C}{\lambda} t_0^{\alpha-1} \ln t_0\,
\int\limits_{c_0\, t_0}^\infty e^{-\xi} \xi^{-\alpha}d\xi+ \frac{C}{\lambda^3} t_0^{\alpha-1} \ln t_0\,
\int\limits_{c_0\, t_0}^\infty e^{-\xi} \xi^{2-3\alpha}d\xi\leq
\]
\[
\leq \frac{C}{\lambda} t_0^{-1} \ln t_0\, e^{-c_0 t_0}+ \frac{C}{\lambda^3} t_0^{-2\alpha-1} \ln t_0\,
\Gamma(3).
\]
Therefore, for a sufficiently  large $t_0$ we have
\begin{equation}\label{Jinfty}
	J_\infty=o(1)\,	\frac{1}{\lambda} t_0^{-\alpha-1} \ln t_0.
\end{equation}

Simple calculations show
\[
|I_{2, \infty}|\leq \frac{C}{\lambda} t_0^{\alpha-1}\,
\int\limits_{c_0 t_0}^\infty e^{-\xi}\, \xi^{-\alpha} (|\ln \xi|+1) d\xi\leq \frac{C}{\lambda} t_0^{-1}\,
\int\limits_{c_0 t_0}^\infty e^{-\xi}\, \xi d\xi.
\]
If we use Lemma \ref{gamma} to estimate the last integral, then for a sufficiently large $t_0$ we obtain
\begin{equation}\label{I2infty}
	I_{2, \infty}=O(1)\,	\frac{1}{\lambda} t_0^{-\alpha-1}.
\end{equation}

Since $|f(\xi, t_o, \alpha)|\leq \xi+\lambda (\gamma +1) \xi^\alpha$, then Lemma \ref{binftyestimate} implies
\[
|I_{4, \infty}|\leq \frac{C}{\lambda} t_0^{\alpha-1}\,
\int\limits_{c_0 t_0}^\infty e^{-\xi}\, \xi^{-\alpha} (|\ln \xi|+1)(\lambda^{-1} \xi^{1-\alpha}+1) d\xi\leq \frac{C}{\lambda} t_0^{-1}\,
\int\limits_{c_0 t_0}^\infty e^{-\xi}\, \xi d\xi.
\]
Again, by virtue of Lemma \ref{gamma}, for sufficiently large $t_0$ we have
\begin{equation}\label{I4infty}
	I_{4, \infty}=O(1)\,	\frac{1}{\lambda} t_0^{-\alpha-1}.
\end{equation}

Finally, consider the integral
\[
I_{5,\infty}=-t_0^{\alpha-1} \,
\int\limits_{c_0 t_0}^\infty e^{-\xi} b_{\alpha,1}(\lambda, t_0, \xi)\frac{2\, g(\xi, \alpha)\lambda\gamma \xi^\alpha [\ln \xi \sin \alpha\pi +\pi \cos \alpha\pi]}{f^2(\xi, t_0, \alpha)+g^2(\xi, \alpha)} d\xi.
\]
Apply Lemma \ref{binftyestimate} to obtain
\[
|I_{5,\infty}|\leq \frac{C}{\lambda}t_0^{\alpha-1} \,
\int\limits_{c_0 t_0}^\infty e^{-\xi} \xi^{-\alpha} (|\ln \xi|+1) d\xi.
\]
Therefore, for sufficiently large $t_0$ one has
\begin{equation}\label{I4infty}
	I_{5, \infty}=O(1)\,	\frac{1}{\lambda} t_0^{-\alpha-1}.
\end{equation}

By virtue of definition (\ref{derivative})), summing up the estimates of the integrals $J_0,$ $J_\infty$ $I_{j,0}$ and $I_{j, \infty}$, we obtain the assertion of the main lemma.

\section{Theorem proofs}
This Rayleigh-Stokes problem is solved, for example, in work \cite{Bazh} (see also \cite{AshurovVaisova}): for any $\varphi \in H$ the problem has a unique solution
\[
u(t)=\sum\limits_{k=1}^\infty B_\alpha (\lambda_k, t) \varphi_k v_k.
\]
Therefore,
\[
U(t, \alpha)= ||u(t)||^2=\sum\limits_{k=1}^\infty B^2_\alpha (\lambda_k, t) |\varphi_k|^2.
\]

Note, according to Lemma \ref{Bazh}, $B_\alpha (\lambda_k, t)>0$ for all $t>0$ and $k\geq 1$. The Main Lemma proves that the derivative $\partial_\alpha B_\alpha (\lambda_k, t_0)$ is negative for a sufficiently large $t_0$ and for all $k\geq 1$, i.e., it is proved that for all $k\geq 1$ and starting from $t_0\geq T_0$ function $B_\alpha(\lambda_k, t_0)$  decreases with respect to $\alpha$. Therefore, Theorem \ref{norm} is a consequence of the Main Lemma.

Let us proceed to the proof of Theorem \ref{main}. By the conditions of the theorem (see (\ref{U_con}), let the given number $d_0$ be such that
$$
\inf_{\alpha\in (0,1)}U(t_0, \alpha)\leq d_0\leq 	\sup_{\alpha\in (0,1)}U(t_0, \alpha), \quad t_0\geq T_0.
$$
Considering that, according to Theorem \ref{norm}, the function $U(t_0, \alpha)$ is strictly monotonic, this immediately implies the existence of a single number $\alpha$ satisfying the additional condition (\ref{ad_con}). Obviously, if the opposite inequalities hold, then the number $d_0$ does not enter the range of the function $U(t_0, \alpha)$, and as a result, the number $\alpha$ satisfying the additional condition (\ref{ad_con}) does not exist.

After the parameter $\alpha$ is found, the solution $u(t)$ that satisfies all the requirements of Definition \ref{def} is found in the same way as in work \cite{AshurovVaisova}.

We turn to the proof of the uniqueness of the solution of the inverse problem (\ref{probAbstract})-(\ref{ad_con}). Let there be two pairs of solutions $\{u_1, \alpha_1\}$ and $\{u_2 , \alpha_2\}$ such that $0<\alpha_j < 1$ and
\begin{equation}\label{eq1}
	\partial_t u_j(t)  + (1+\gamma\, \partial_t^{\alpha_j})A u_j(t) = 0,\quad 0< t \leq T,
\end{equation}
\begin{equation}\label{in1}
	u_j(0)= \varphi,
\end{equation}
where $j=1, 2$. For $t_0\geq T_0$ additional condition (\ref{ad_con}) implies
\begin{equation}\label{ad_con2}
	||u_1(t_0)||^2=||u_2(t_0)||^2=d_0.
\end{equation}

Consider the following functions
$$
T^j_k(t)=(u_j(t), v_k) \quad j=1,2; \quad k=1,2, \cdot\cdot\cdot
$$
Then equations (\ref{eq1}) and (\ref{in1}) imply
$$
\partial_t T^j_k(t)  + \lambda_k (1+\gamma\, \partial_t^{\alpha_j}) T^j_k(t) = 0 \quad  w^j_k(0) =
\varphi_k.
$$
Solutions to these Cauchy-type problems can be represented as (see (\ref{BCauchySolutionIN})
\[
T_k^j(t)=\varphi_k B_{\alpha_j}(\lambda_k, t).
\]
Note that, by virtue of Lemma \ref{Bazh}, the functions $ B_{\alpha_j}(\lambda_k, t)$ are positive for all $k$ and $j$, $t>0$. Assume, for example, that $\alpha_1<\alpha_2$ and $t_0\geq T_0$. Then the Main Lemma implies
\[
|\varphi_k|^2 B^2_{\alpha_1}(\lambda_k, t_0)>|\varphi_k|^2 B^2_{\alpha_2}(\lambda_k, t_0),
\]
for all $k\geq 1$. Then 	$||u_1(t_0)||^2>||u_2(t_0)||^2$ and this contradicts equality (\ref{ad_con2}). Therefore, $\alpha_1=\alpha_2$ and hence $T_k^1(t)\equiv T_k^2(t)$. Hence
$$
(u_1(t)-u_2(t), v_k)=0
$$
for all $k$. Finally, from completeness of the set of eigenfunctions $\{v_k\}$ in $H$, we have $u_1(t)\equiv u_2(t)$. Hence Theorem
\ref{main} is completely proved.

\section{Examples of the $A$ operator and various additional conditions}

1. The use of specific self-adjoint operators as $A$ allows us to consider various mathematical models. As an example, we can take the operators of mathematical physics considered in section 6 of the article by M. Ruzhansky et al. \cite{Ruz}, including the classical Sturm-Liouville problem, differential models with involution, fractional Sturm-Liouville operators, harmonic and anharmonic oscillators, Hamiltonians Landau, fractional Laplacians, harmonic and anharmonic operators on the Heisenberg group. As a result, we obtain various Rayleigh-Stokes equations for which Theorems \ref{norm} and \ref{main} of this work will be valid.

It should be noted that the authors of \cite{Ruz} considered inverse problems of recovering the right-hand side of the subdiffusion equation with Caputo derivatives, taking various positive operators with discrete spectrum as the elliptic part.

2. If we take $\mathbb{R}^N$ as the Hilbert space $H$ and consider the $N$-dimensional symmetric quadratic matrix $A=\{a_{i,j}\}$  with constant elements $a_{ i,j}$ as the operator $A$ then the problem (\ref{probAbstract}), (\ref{ad_con}) coincides with the inverse problem for the linear system of Rayleigh-Stokes ordinary differential equations.

3. Now let us discuss what other additional conditions can be considered instead of (\ref{ad_con}). Let $\Phi: R_+\rightarrow R$ be a continuous function
such that, the domain of definition $D(\Phi(A))$ of the operator
\[
\Phi(A)u(t)=\sum\limits_{k=1}^\infty \Phi(\lambda_k)\, (u(t),v_k)
v_k,
\]
is a subset of $D(A)$. Then as an additional condition we may take the
following equation
\begin{equation}\label{ex}
	U(t_0, \alpha) \equiv ||\Phi(A) u(t_0)||^2=d_0,
\end{equation}
where $t_0\geq T_0$ is a fixed time instant. It is not hard to verify,
that for this additional information both Theorems \ref{norm} and \ref{main} hold true.

Let us consider some particular cases of the function $\Phi(\lambda)$. If $\Phi(\lambda) \equiv 1$ then the corresponding value
\[
U(t_0, \alpha)= ||u(t_0)||^2
\]
coincides with (\ref{ad_con}). In case when $\Phi(\lambda) =  \lambda$, we have $U(t_0, \alpha)=
||Au(t_0)||^2 $.

\section{Conclusion}

The well-known Rayleigh-Stokes problem (\ref{probIN}) involves the fractional derivative $\partial_t^\alpha$, which describes the behavior of a viscoelastic flow. However, this parameter is often unknown and difficult to measure directly. Therefore, it is undoubtedly topical to study the inverse problem to determine this physical quantity from some indirectly observed information about the solutions. The present work is devoted to the study of this new inverse problem for the Rayleigh-Stokes equation. In this paper, along with other results, we prove that the additional condition $||u(x, t_0)||_{L_2(\Omega)}=d_0$ for sufficiently large $t_0$ uniquely determines the parameter $\alpha$.

Another very important problem studied in this work is to find out the dependence of the behavior of the solution of the initial-boundary value problem on the order of the fractional derivative. In this work, an interesting fact was discovered: if we consider the norm of the solution $||u(x, t_0)||^2_{L_2(\Omega)}$ as a function of the parameter $\alpha$, then this is a decreasing function. In other words, the norm acquires its maximum value when the order of the fractional derivative is close to zero, and its minimum value - when this parameter is close to one.

\end{document}